\documentclass[11pt]{amsart}
\usepackage{latexsym}
\usepackage{palatino}
\usepackage{amsmath, amssymb}
\usepackage[dvipsnames,usenames]{color}
\usepackage[all]{xy}
\newtheorem{dfs}{Definition}[section]

\newtheorem{thms}[dfs]{Theorem}
\newtheorem{props}[dfs]{Proposition}

\newtheorem{cors}[dfs]{Corollary}

\newtheorem{conjs}[dfs]{Conjecture}

\newtheorem*{thm*}{Theorem}

\title[Characterizing classifiable AH algebras]
{Characterizing classifiable AH algebras}

\address{Department of Mathematics \\ Purdue University \\ 150 N. University St. \\ West Lafayette IN \\ USA \\ 47907 }

\email{atoms@purdue.edu}

\begin{document}

\maketitle

\begin{quote}
{\bf Abstract.}  We observe almost divisibility for the original Cuntz semigroup of a simple AH algebra with strict comparison.  As a consequence, the properties of strict comparison, finite nuclear dimension, and $\mathcal{Z}$-stability are equivalent for such algebras, confirming partially a conjecture of Winter and the author.  \\

\vspace{3mm}
\noindent
{\bf R\'esum\'e.}   Nous observons presque-divisibilit\'e pour le semigroupe de Cuntz original d'un alg\`ebre AH simple avec la comparaison stricte. Comme cons\'equence, les propri\'et\'es de comparaison stricte, la dimension nucl\'eaire finie et la Z-stabilit\'e sont \'equivalentes pour un tel alg\`ebre, confirmant partiellement une conjecture de Winter et l'auteur.
\end{quote}

\section{introduction}
Kirchberg proved in 1994 that tensorial absorption of $\mathcal{O}_\infty$ and pure infiniteness were equivalent for simple separable nuclear C$^*$-algebras.  This theorem can be reinterpreted as the equivalence of $\mathcal{Z}$-stability and strict comparison for simple separable nuclear traceless C$^*$-algebras (see \cite{Ro}), an equivalence that can be considered in stably finite algebras, too.  This and the results of \cite{TW} led to the following conjecture:

\begin{conjs}[T-Winter, 2008]\label{conjecture}
Let $A$ be a simple unital nuclear separable C$^*$-algebra.  The following are equivalent:
\begin{enumerate}
\item[(i)] $A$ has finite nuclear dimension;
\item[(ii)] $A$ is $\mathcal{Z}$-stable;
\item[(iii)] $A$ has strict comparison of positive elements.
\end{enumerate}
\end{conjs}

\noindent
It is expected that these conjecturally equivalent conditions will characterize those algebras which are determined up to isomorphism by their Elliott invariants.  In the absence of even the weakest condition, (iii), one cannot classify AH algebras using only the Elliott invariant (\cite{T3}).  While it is possible that this can be corrected with an enlarged invariant, the jury is still out.  Combining the main result of \cite{W} with that of \cite{R} yields (i) $\Rightarrow$ (ii), while R\o rdam proves (ii) $\Rightarrow$ (iii) in \cite{Ro}.  This note yields the following result.
\begin{thms}\label{main}
Conjecture \ref{conjecture} holds for AH algebras.
\end{thms}
\noindent
We proceed by proving (ii) implies (iii) for AH algebras (Corollary \ref{maincor}) and appealing to \cite[Corollary 6.7]{W}.  It should be noted that our contribution is quite modest:  we simply verify the hypotheses of the main result of \cite{W}.

\section{Strict comparison and almost divisibility}

Let $A$ be a unital C$^*$-algebra, and $\mathrm{T}(A)$ its simplex of tracial states.  Let $\mathrm{Aff}(\mathrm{T}(A))$ denote the continuous $\mathbb{R}$-valued affine functions on $\mathrm{T}(A)$, and let $\mathrm{lsc}(\mathrm{T}(A))$ denote the set of bounded lower semicontinuous strictly positive affine functions on $\mathrm{T}(A)$.  Set $\mathrm{M}_{\infty}(A) = \cup_n \mathrm{M}_n(A)$.  For positive 
$a,b \in \mathrm{M}_{\infty}(A)$, we write $a \precsim b$ if there is a sequence $(v_n)$ in $\mathrm{M}_{\infty}(A)$ such that $v_nbv_n^* \to a$ in norm (for the norm, view $\mathrm{M}_\infty(A)$ sitting naturally inside $A \otimes \mathcal{K}$).  We write $a \sim b$ if $a \precsim b$ and $b \precsim a$.  Set $W(A) = \{a \in \mathrm{M}_\infty(A) \ | \ a \geq 0 \}/\sim$, and let $\langle a \rangle$ denote the equivalence class of $a$.  $W(A)$ can be made into an ordered Abelian monoid by setting
\[
\langle a \rangle + \langle b \rangle = \langle a \oplus b \rangle \ \ \mathrm{and} \ \ \langle a \rangle \leq \langle b \rangle \Leftrightarrow a \precsim b.
\]
$W(A)$ is the original {\it Cuntz semigroup} of $A$.  

We say that $W(A)$ is {\it almost divisible} if for any $x \in W(A)$ and $n \in \mathbb{N}$, there exists $y \in W(A)$ such that 
\[
ny \leq x \leq (n+1)y.
\]
If $\tau \in \mathrm{T}(A)$, we define $d_\tau: W(A) \to \mathbb{R}^+$ by
\[
d_\tau(\langle a \rangle) = \lim_{n \to \infty} \tau(a^{1/n}).
\]
This map is known to be well-defined, additive, and order preserving, and for a fixed positive $a \in \mathrm{M}_{\infty}(A)$, $A$ simple, the map $\tau \mapsto d_\tau(a)$ belongs to $\mathrm{lsc}(\mathrm{T}(A))$.  If $a \precsim b$ whenever $d_\tau(a) < d_\tau(b)$ for every $\tau \in \mathrm{T}(A)$, then we say that $A$ has {\it strict comparison}.  

It is implicit in Conjecture \ref{conjecture} that a unital simple separable nuclear C$^*$-algebra with strict comparison of positive elements should have almost divisible Cuntz semigroup, but no general method has yet been found to establish this fact.  Positive results have been limited to particular classes of C$^*$-algebras.  Here we handle the case of AH algebras.

\begin{props}\label{dense}
Let $A$ be a unital, simple, stably finite C$^*$-algebra with strict comparison of positive elements.  Suppose that for any $f \in \mathrm{Aff}(\mathrm{T}(A))$ and $\epsilon>0$ there is positive $a \in \mathrm{M}_\infty(A)$ such that
\[
|f(\tau) - d_\tau(a)| < \epsilon, \ \forall \tau \in \mathrm{T}(A).
\]
It follows that $W(A)$ is almost divisible.
\end{props}

\begin{proof}
Let $g \in \mathrm{lsc}(\mathrm{T}(A))$ be given.  Then there is a strictly increasing sequence $(f_i)$ of strictly positive functions in $\mathrm{Aff}(\mathrm{T}(A))$ with the property that $\sup_i f_i(\tau) = g(\tau)$.  The function $f_i - f_{i-1}$ is strictly positive and continuous, and so achieves a minimum value $\epsilon_i >0$ on the compact set $\mathrm{T}(A)$.  Passing to a subsequence, we may assume that $\epsilon_i < \epsilon_{i-1}$.  By hypothesis, we can find, for each $i$, a positive $a_i \in \mathrm{M}_\infty(A)$ such that
\[
|f_i(\tau)-d_\tau(a_i)|<\epsilon_{i+1}/3.
\]
It follows that $(\tau \mapsto d_\tau(a_i))_{i \in \mathbb{N}}$ is a strictly increasing sequence in $\mathrm{lsc}(\mathrm{T}(A))$ with supremum
$g$.  By strict comparison, we have $a_i \precsim a_{i+1}$, i.e., $(\langle a_i \rangle)_{i \in \mathbb{N}}$ is an increasing sequence in $W(A)$.  \cite[Theorem 1]{cei} then guarantees the existence of a supremum $y$ for this sequence in $W(A \otimes \mathcal{K}) \supseteq W(A)$.  The map $d_\tau$ is supremum preserving for each $\tau$, and we conclude that $d_\tau(y) = g(\tau), \ \forall \tau \in \mathrm{T}(A)$.  

Now let $x \in W(A)$ and $n \in \mathbb{N}$ be given, and set $h(\tau) = d_\tau(x)$ for each $\tau \in \mathrm{T}(A)$.  It is straightforward to find $g \in \mathrm{lsc}(\mathrm{T}(A))$ with the property that 
\[
ng < h < (n+1)g.
\]
We may moreover find $x \in W(A \otimes \mathcal{K})$ such that $d_\tau(x) = g(\tau)$, as in the first part of the proof.  By strict comparison, any representative $a$ for $x$ (that is, $\langle a \rangle =x$ in $W(A \otimes \mathcal{K})$) will satisfy
\[
n\langle a \rangle \precsim x \precsim (n+1)\langle a \rangle,
\]
so it remains only to prove that $a$ can be chosen to lie in $\mathrm{M}_\infty(A)$, rather than $A \otimes \mathcal{K}$.  Let $\mathbf{1_k}$ denote the unit of $\mathrm{M}_k(A)$.  Since $g$ is bounded, we have $g(\tau) < k = d_\tau(\mathbf{1_k})$ for some $k \in \mathbb{N}$ and all $\tau$.  By strict comparison, then, $\langle a \rangle$ is dominated by a Cuntz class in $W(A)$.  It then follows from \cite[Theorem 4.4.1]{5guys}  that there is positive $b \in \mathrm{M}_\infty(A)$ such that $\langle b \rangle = \langle a \rangle$, completing the proof.
\end{proof}

\begin{cors}\label{maincor}
Let $A$ be a unital simple AH algebra with strict comparison.  It follows that $A$ is $\mathcal{Z}$-stable.
\end{cors}

\begin{proof}
$A$ is stably finite, and satisifes the hypothesis of Proposition \ref{dense} concerning the existence of suitable $a$ for each $f$ and $\epsilon$ by \cite[Theorem 5.3]{bpt}.  It therefore has strict comparison and almost divisible Cuntz semigroup.  These hypotheses, together with the fact that $A$ is simple, nuclear, separable, and has locally finite nuclear dimension allow us to appeal to \cite[Theorem 6.1]{W} and conclude that $A$ is $\mathcal{Z}$-stable.
\end{proof}

We must concede that Proposition \ref{dense} closes a gap in the proof of \cite[Theorem 1.2]{T2}.  There, we proved that the hypotheses of Propostion \ref{dense} were satisfied for a simple unital ASH algebra with slow dimension growth but neglected to explain how this guarantees almost divisibility for $W(A)$ as opposed to $W(A \otimes \mathcal{K})$.  While this could have been done in several ways, our appeal here to the recent article \cite{5guys} was the most efficient one.

\end{document}